\documentclass[12pt]{article}
\usepackage{algorithmic,amsmath,amsthm,amssymb, graphicx, url, algorithm2e, wasysym}
\usepackage[margin=1in]{geometry}

\title{Towards A Deeper Geometric, Analytic and Algorithmic Understanding of Margins}

\author{
Aaditya Ramdas \\
Machine Learning Department\\
Carnegie Mellon University\\
\texttt{aramdas@cs.cmu.edu}
\and 
Javier Pe\~{n}a\\
Tepper School of Business\\
Carnegie Mellon University\\
\texttt{jfp@andrew.cmu.edu}
}

\usepackage{color}

\newtheorem{theorem}{Theorem}

\newtheorem{proposition}{Proposition}

\newtheorem{corollary}{Corollary}

\newcommand{\half}{\tfrac1{2}}

\newcommand{\R}{\mathbb{R}}

\newcommand{\ones}{\mathbf{1}}
\newcommand{\zeros}{\mathbf{0}}

\begin{document}

\maketitle

\begin{abstract}
Given a matrix $A$, a linear feasibility problem (of which linear classification is a special case) aims to find a solution to a primal problem $w: A^Tw > \zeros$ or a certificate for the dual problem which is a probability distribution $p: Ap = \zeros$. Inspired by the continued importance of ``large-margin classifiers'' in machine learning, this paper studies a condition measure of $A$ called its \textit{margin} that determines the difficulty of both  the above problems.  
To aid geometrical intuition, we first establish new characterizations of the margin in terms of relevant balls, cones and hulls. Our second contribution is analytical, where we present generalizations of Gordan's theorem, and variants of Hoffman's theorems, both using margins. We end by proving some new results on a classical iterative scheme, the Perceptron, whose convergence rates famously depends on the margin. Our results are relevant for a deeper understanding of margin-based learning and proving convergence rates of iterative schemes, apart from providing a unifying perspective on this vast topic.
\end{abstract}


\section{Introduction}
Assume that we have a $d \times n$ matrix $A$ representing $n$ points $a_1,...,a_n$ in $\R^d$. In this paper, we will be concerned with linear feasibility problems that ask if there exists a vector $w \in \R^d$ that makes positive dot-product with every $a_i$, i.e.
\begin{equation}\label{eqP}
?\exists w ~:~ A^T w > \zeros, \tag{P}
\end{equation}
where boldfaced $\zeros$ is a vector of zeros. The corresponding algorithmic question is ``if (P) is feasible, how quickly can we find a 
$w$ that demonstrates (P)'s feasibility?''.

 Such problems abound in optimization as well as machine learning. For example, consider \textit{binary linear classification} - given $n$ points $x_i \in \R^d$ with labels $y_i \in \{+1,-1\}$, a classifier $w$ is said to separate the given points if $w^T x_i$ has the same sign as $y_i$ or succinctly $y_i (w^T x_i) > 0$ for all $i$. Representing $a_i = y_i x_i$ shows that this problem is a specific instance of (P).

 We call (P) the \textit{primal} problem, and (we will later see why) we define the \textit{dual} problem (D) as:
\begin{equation}\label{eqD}
?\exists p \geq \zeros ~:~ Ap = \zeros, p \ne \zeros, \tag{D}
\end{equation}
and the corresponding algorithmic question is ``if (D) is feasible, how quickly can we find a certificate $p$ that demonstrates feasibility of (D)?''.\\

Our aim is to deepen the geometric, algebraic and algorithmic understanding of the problems (P) and (D),  tied together by a concept called \textit{margin}. Geometrically, we provide intuition about ways to interpret margin in the primal and dual settings relating to various balls, cones and hulls. Analytically, we prove new margin-based versions of classical results in convex analysis like Gordan's and Hoffman's theorems. Algorithmically, we give new insights into the classical Perceptron algorithm. We begin with a gentle introduction to some of these concepts, before getting into the details.

\paragraph{\textbf{Notation}} 
When we write $v \geq w$ for vectors $v,w$, we mean $v_i \geq w_i$ for all their indices $i$ (similarly $v\leq w, v=w$). 
To distinguish surfaces and interiors of balls more obviously to the eye in mathematical equations, we choose to denote Euclidean balls in $\mathbb{R}^d$ by $\Circle := \{w \in \mathbb{R}^d : \|w\|=1\}$, $\CIRCLE:=\{w \in \mathbb{R}^d : \|w\|\leq 1\}$ and the probability simplex $\mathbb{R}^n$ by $\triangle:=\{p \in \mathbb{R}^n : p\geq \zeros, \|p\|_1 = 1\}$. 
 We denote the linear subspace spanned  by $A$ as lin$(A)$, and convex hull of $A$ by conv$(A)$.
Lastly, define $\CIRCLE_A := \CIRCLE \cap \mathrm{lin}(A)$  and $r\CIRCLE$ is the  ball of radius $r$ ($\Circle_A, r\Circle$ are similarly defined). 

\subsection{\textbf{Margin $\rho$}} The margin of the problem instance $A$ is classically defined as
\begin{eqnarray}
\rho &:=& \sup_{w \in \Circle} \inf_{p \in \triangle}\ w^T Ap \label{eq:margin}\\
&=& \sup_{w \in \Circle} \inf_i \ w^T a_i.\nonumber
\end{eqnarray}

If there is a $w$ such that $A^T w > \zeros$, then $\rho > 0$. If for all $w$, there is a point at an obtuse angle to it, then $\rho < 0$. At the boundary $\rho$ can be zero. The $w \in \Circle$ in the definition is important -- if it were $w \in \CIRCLE$, then $\rho$ would be non-negative, since $w=0$ would be allowed.

This definition of margin was introduced by Goffin~\cite{G80} who gave several geometric interpretations. It has since been extensively studied (for example, \cite{R94,R95} and \cite{CC01}) as a notion of complexity and conditioning of a problem instance. Broadly, the larger its magnitude, the better conditioned the pair of feasibility problems (P) and (D) are, and the easier it is to find a witnesses of their feasibility. Ever since \cite{V98}, the margin-based algorithms have been extremely popular with a growing literature in machine learning which it is not relevant to presently summarize.

In Sec.~\ref{sec:affmargin}, we define an important and ``corrected'' variant of the margin, which we call \textit{affine-margin}, that turns out to be the actual quantity determining convergence rates of iterative algorithms.

\paragraph{\textbf{Gordan's Theorem }} This is a classical \textit{theorem of the alternative}, see \cite{BL06,C83}. It implies that exactly one of (P) and (D) is feasible. 
Specifically, it states that exactly one of the following statements is true.
\begin{enumerate}
\item There exists a $w$ such that $A^T w > \zeros$.
\item There exists a $p \in \triangle$ such that $Ap = \zeros$.
\end{enumerate} 
This, and other separation theorems like Farkas' Lemma (see above references), are widely applied in algorithm design and analysis. 
We will later prove generalizations of Gordan's theorem  using affine-margins.

\paragraph{\textbf{Hoffman's Theorem}}  The classical version of the theorem from \cite{H52} characterizes how close a point is to the solution set of the feasibility problem $Ax\le b$ in terms of the amount of  violation in the inequalities and a problem dependent constant. In a nutshell, if
$\mathbb{S} ~:=~ \{x | Ax \leq b \} \neq \emptyset$ 
then
\begin{equation} \label{eq:hoffintro}
\mathrm{dist}(x,\mathbb{S}) ~\leq~ \tau \big\| [Ax-b]_+ \big\|
\end{equation}
where $\tau$ is the ``Hoffman constant'' and it depends on $A$ but is \textit{independent of $b$}. This and similar theorems have found extensive use in convergence analysis of algorithms - examples include \cite{GPS12,HL12,SP13}.

G\"uler, Hoffman, and Rothblum~\cite{GHR95} generalize this bound to any norms on the left and right hand sides of the above inequality. We will later prove theorems of a similar flavor for (P) and (D), where $\tau^{-1}$ will almost magically turn out to be the affine-margin. 
Such theorems are used for proving rates of convergence of algorithms, and having the constant explicitly in terms of a familiar quantity is useful.

\subsection{\textbf{Summary of Contributions}} 
\begin{itemize}
\item  \textbf{Geometric}: In Sec.\ref{sec:affmargin}, we define the \textit{affine-margin}, and argue why a subtle difference from Eq.(\ref{eq:margin}) makes it the ``right'' quantity to consider, especially for problem (D). 
We then establish geometrical characterizations of the affine-margin when (P) is feasible as well as when (D) is feasible
and connect it to well-known \textit{radius theorems}. This is the paper's appetizer.

\item \textbf{Analytic}: Using the preceding geometrical insights, in Sec.\ref{sec:gordan} we prove two generalizations of Gordan's Theorem to deal with alternatives involving the affine-margin when either (P) or (D) is strictly feasible. Building on this intuition further, in Sec.\ref{sec:hoffman}, we prove several interesting variants of Hoffman's Theorem, which explicitly involve the affine-margin when either (P) or (D) is strictly feasible. This is the paper's main course.

\item \textbf{Algorithmic}: In Sec.\ref{sec:NP}, we prove new properties of the Normalized Perceptron, like its margin-maximizing and margin-approximating property for (P) and dual convergence for (D). This is the paper's dessert.

\end{itemize}

We end with a historical discussion relating Von-Neumann's and Gilbert's algorithms, and their advantage over the Perceptron.

\section{From Margins to \textit{Affine}-Margins}\label{sec:affmargin}
An important but subtle point about margins is that the  quantity  determining the difficulty of solving (P) and (D) is actually \textit{not} the margin as defined classically in Eq.(\ref{eq:margin}), but the affine-margin which is the margin when $w$ is restricted to lin($A$), i.e. $w=A\alpha$ for some coefficient vector $\alpha \in \R^n$. The affine-margin is defined as
\begin{eqnarray}
\rho_A &:=& \sup_{w \in \Circle_A} \inf_{p \in \triangle}\ w^T Ap \nonumber \\
&=& \sup_{\|\alpha\|_G=1} \inf_{p \in \triangle} \alpha^T Gp \label{affmargin}
\end{eqnarray}
where $G=A^TA$ is a key quantity called the Gram matrix, and $\|\alpha\|_G = \sqrt{\alpha^T G \alpha}$ is easily seen to be a self-dual semi-norm. 

Intuitively, when the problem (P) is infeasible but  $A$ is not full rank, i.e. lin($A$) is not $\R^d$, then $\rho$ will never be negative  (it will always be zero), because one can always pick $w$ as a unit vector perpendicular to  lin$(A)$, leading to a zero dot-product with every $a_i$. Since no matter how easily inseparable $A$ is, the margin is always zero if $A$ is low rank, this definition does not capture the difficulty of verifying linear infeasibility.

Similarly, when the problem (P) is feasible, it is easy to see that searching for $w$ in directions perpendicular to $A$ is futile, and one can restrict attention to lin$(A)$, again making this the right quantity in some sense. For clarity, we will refer to 
\begin{equation}
\rho_A^+ ~:=~ \max\{0,\rho_A\} \mbox{\ \ ; \ \ } \rho_A^- ~:=~ \min\{0,\rho_A\}
\end{equation} 
when the problem (P) is strictly feasible ($\rho_A>0$) or strictly infeasible ($\rho_A<0$) respectively.

We remark that when $\rho > 0$, we have $\rho_A^+ =  \rho_A = \rho$, so the distinction really matters when $\rho_A < 0$, but it is still useful to make it explicit. One may think that if $A$ is not full rank, performing PCA would get rid of the unnecessary dimensions. However, we often wish to only perform elementary operations on (possibly large matrices) $A$ that are much simpler than eigenvector computations.

\subsubsection*{\textbf{Instability of $\rho_A^-$ compared to $\rho$}}
 Unfortunately, the behaviour of $\rho_A^-$ is quite finicky -- unlike $\rho_A^+$ it is not stable to small perturbations when conv($A$) is not full-dimensional. To be more specific, if (P) is strictly feasible and we perturb all the vectors by a small amount or add a vector that maintains feasibility, $\rho_A^+$ can only change by a small amount. However, if (P) is strictly \textit{in}feasible and we perturb all the vectors by a small amount or add a vector that maintains infeasibility, $\rho_A^-$ can change by a large amount. 
 
For example, assume lin$(A)$ is not full-dimensional, and $|\rho_A^-|$ is large. If we add a new vector $v$ to $A$ to form $A' = \{A \cup v\}$ where $v$ has a even a tiny component $v^\perp$ orthogonal to lin($A$), then $\rho_{A'}^-$ suddenly becomes zero. This is because it is now possible to choose a vector $w = v^\perp/\|v^\perp\|$ which is in lin$(A')$, and makes zero dot-product with $A$, and positive dot-product with $v$. Similarly, instead of adding a vector, if we perturb a given set of vectors so that lin($A$) increases dimension, the negative margin can suddenly jump from to zero. 

Despite its instability and lack of ``continuity'', it is indeed this negative affine margin that determines rate of convergence of algorithms for (D). 
In particular, the convergence rate of the von Neumann--Gilbert algorithm for (D) is determined by $\rho_A^-$ much the same way as the convergence rate of the perceptron algorithm for (P) is determined by $\rho_A^+$.  We discuss these issues in detail in Section~\ref{sec:NP} and Section~\ref{sec.VNG}.

\subsection{Geometric Interpretations of $\rho_A^+$}
\label{sec:rhoaplus}

The positive margin has many known geometric interpretations -- it is the width of the feasibility cone, and also the largest ball centered on the unit sphere that can fit inside the dual cone ($w : A^T w > \zeros$ is the dual cone of cone$(A)$) -- see, for example \cite{CC01, FV99}. Here, we provide a few more interpretations. Remember that $\rho_A^+ = \rho$ when Eq.\eqref{eqP} is feasible.

\begin{proposition}\label{margindual}
The distance of the origin to conv$(A)$ is $\rho^+_A$.
\begin{equation}\label{dist.origin}
\rho_A^+ ~=~ \inf_{p \in \triangle} \|p\|_G ~=~ \inf_{p \in \triangle} \|Ap\|
\end{equation}
\end{proposition}
\begin{proof} When $\rho_A \le 0$, $\rho_A^+ = 0$ and Eq.\eqref{dist.origin} holds because (D) is feasible making the right hand side also zero.  When $\rho_A > 0$, 
\begin{equation}\label{eq:dual}
\rho_A^+ = \sup_{w \in \Circle} \inf_{p \in \triangle} w^T A  p = \sup_{w \in \CIRCLE} \inf_{p \in \triangle} w^T A  p = \inf_{p \in \triangle} \sup_{w \in \CIRCLE} w^T A  p =  \inf_{p \in \triangle} \|Ap\|.
\end{equation}
Note that the first two equalities holds when $\rho_A > 0$, the next by the minimax theorem, and the last by self-duality of $\|.\|$. \end{proof}

The quantity $\rho_A^+$ is also closely related to a particular instance of the Minimum Enclosing Ball (MEB) problem. While it is  common knowledge that MEB is connected to margins (and support vector machines), it is possible to explicitly characterize this relationship, as we have done below.

\begin{proposition}\label{meb} Assume $A = \left[\begin{array}{ccc}a_1 & \cdots & a_n \end{array} \right] \in \R^{d\times n}$ and $\|a_i\|=1,\; i=1,\dots,d$.  Then
the radius of the minimum enclosing ball of conv($A$) is $ \sqrt{1 - \rho_A^{+2}}$.
\end{proposition}
\begin{proof} 
It is a simple exercise to show that the following are the MEB problem, and its Lagrangian dual 
\begin{eqnarray}
\min_{c,r}& & r^2 \mbox{\ \ \ s.t. \ \ \ } \|c-a_i\|^2 \leq r^2 \notag
\\ 
\label{eq:MEB}
\\ \notag
\max_{p \in \triangle}& & \ 1 - \|Ap\|^2. \notag
\end{eqnarray}
The result then follows from  Proposition~\ref{margindual}. 
\end{proof}

As we show in Section~\ref{sec:NP}, the (Normalized) Perceptron and related algorithms that we introduce later yields a sequence of iterates that converge to the center of the MEB, and if the distance of the origin to conv($A$) is zero (because $\rho_A < 0$), then the sequence of iterates coverges to the origin, and the MEB just ends up being the unit ball.

\subsection{Geometric Interpretations of $|\rho_A^-|$}\label{sec:rhoaminus}

\begin{proposition}\label{thm:radthm} If $\rho_A \le 0$ then
$|\rho_A^-|$ is the radius of the largest Euclidean ball centered at the origin that completely fits inside the \textit{relative interior} of the convex hull of $A$. Mathematically,
\begin{eqnarray}\label{eq:rhoaminus}
|\rho_A^-| &=&\sup \Big \{ R \ \big| \|\alpha\|_G \leq R \Rightarrow A\alpha \in \mathrm{conv}(A) \Big \}.
\end{eqnarray}
\end{proposition}
\begin{proof}
We split the proof into two parts, one for each inequality.

\noindent
\textbf{(1) For inequality $\geq$.} Choose any $R$ such that $A\alpha \in \mathrm{conv}(A)$ for any $\|\alpha\|_G \leq R$. Given an arbitrary $\|\alpha'\|_G=1$, put $\tilde{\alpha}:= -R \alpha'$. By our assumption on $R$, since $\|\tilde{\alpha}\|_G = R$, we can infer that $A\tilde{\alpha} \in \mathrm{conv}(A)$ implying there exists a $\tilde{p} \in \triangle$ such that  $A\tilde{\alpha} = A\tilde{p}$. Also
\[
\alpha'^TG\tilde{p} ~=~  \alpha'^TG\tilde{\alpha}~=~ -R \|\alpha'\|_G^2 ~=~ -R.
\]
Thus $\inf_{p \in \triangle} \ \alpha'^TGp ~\leq~ -R.$ Since this holds for any $\|\alpha'\|_G = 1$, it follows that  $$\sup_{\|\alpha\|_G=1}\inf_{p \in \triangle}  \alpha^TGp ~\leq~ -R.$$ In other words, $|\rho_A^-| ~\geq~ R.$

\medskip

\noindent
\textbf{(2) For inequality $\leq$.} It suffices to show $\|\alpha\|_G \leq |\rho_A^-| \Rightarrow A\alpha \in \mathrm{conv}(A)$. We will prove the contrapositive $A\alpha \notin \mathrm{conv}(A) \Rightarrow \|\alpha\|_G > |\rho_A^-|$.
Since $\mathrm{conv}(A)$ is closed and convex, if $A\alpha \notin \mathrm{conv}(A)$, then there exists a hyperplane separating $A\alpha$ and $\mathrm{conv}(A)$ in lin$(A)$.
 That is, there exists $(\beta, b)$ with  $\|A\beta\|=1$ in  lin($A$) and a constant $b \in \R$ such that $\beta^T A^T A \alpha = \beta^T G \alpha < b$  
 and $\beta^T A^T A p = \beta^T G p \ge b$ for all $p \in \triangle$.   In particular,
\[
\beta^T G \alpha < \inf_{p \in \triangle} \beta^T G p \le \sup_{\|\beta\|_G=1}\inf_{p \in \triangle} \beta^T G p =  \rho_A^-.
\]
Since $\rho_A^- \le 0$, it follows that $|\rho_A^-| < |\beta^T G \alpha|
\leq \|\beta\|_G \|\alpha\|_G = \|\alpha\|_G.$
\end{proof}

One might be tempted to deal with the usual margin and prove that
\begin{equation}\label{eq:oldrhoaminus}
|\rho| ~=~ \sup \Big \{ R ~\big |~ \|w\|\leq R \Rightarrow w \in \mathrm{conv}(A) \Big \}
\end{equation}
While the two definitions are equivalent for full-dimensional lin$(A)$, they differ when lin$(A)$ is not full-dimensional, which is especially relevant in the context of infinite dimensional reproducing kernel Hilbert spaces, but could even occur when $A$ is low rank. In this case, Eq.(\ref{eq:oldrhoaminus}) will always be zero since a full-dimensional ball cannot fit inside a finite-dimensional hull. The right thing to do is to only consider balls ($\|\alpha\|_G \leq R$) in the linear subspace spanned by columns of $A$ (or the relative interior of the convex hull of $A$) and not full-dimensional balls ($\|w\|\leq R$). The reason it matters is that it is this altered $|\rho_A^-|$ that determines rates for algorithms and the complexity of problem (D), and not the classical margin in Eq.(\ref{eq:margin}) as one might have expected.

\subsubsection*{``Radius Theorems''} 
Recall that $A\triangle = \{Ap : p \in \triangle \} = $  conv$(A)$, $\CIRCLE_A = \CIRCLE \cap \mbox{lin}(A)$, and $R\CIRCLE_A$ is just $\CIRCLE_A$ of radius R. Since $\|\alpha\|_G\leq R ~\Leftrightarrow~ \|A\alpha\| \leq R ~\Leftrightarrow~ A\alpha \in R\CIRCLE_A$, an enlightening restatement of Eq.\eqref{eq:rhoaminus} and Eq.\eqref{eq:oldrhoaminus}  is
\begin{eqnarray*}
|\rho_A^-| = \sup \Big \{R \ \big|\ R\CIRCLE_A \subseteq A\triangle \Big \} \mbox{, ~and~ }
|\rho| = \sup \Big \{R \ \big|\ R\CIRCLE \subseteq A\triangle \Big \}.
\end{eqnarray*}
 It can be read as ``largest radius (affine) ball that fits inside the convex hull''. There is a nice parallel to the smallest (overall) and smallest positive singular values of a matrix. Using $A\CIRCLE = \{Ax : x\in \CIRCLE\}$ for brevity,
\begin{eqnarray}
\sigma^+_{\min}(A) = \sup \Big \{R \ \big|\ R\CIRCLE_A \subseteq A\CIRCLE \Big \} \mbox{, ~and~ }
   \sigma_{\min}(A) = \sup \Big \{R \ \big|\ R\CIRCLE \subseteq A\CIRCLE \Big \}
\end{eqnarray}
This highlights the role of the margin is a measure of 
conditioning of the linear feasibility systems (P) and (D).  Indeed, there are a number of far-reaching extensions of the classical ``radius theorem'' of  \cite{EY36}.  The latter states that the Euclidean distance from a square non-singular matrix $A \in \R^{n \times n}$ to the set of singular matrices in $\R^{n \times n}$ is precisely $\sigma_{\min}(A)$.  In an analogous fashion, for the feasibility problems (P) and (D), the set $\Sigma$ of {\em ill-posed} matrices $A$ are those with $\rho =0$. Cheung and Cucker~\cite{CC01} show that for a given a matrix $A \in \R^{m\times n}$ with normalized columns, the margin is the largest perturbation of a row to get an ill-posed instance or the ``distance to ill-posedness'', i.e.
\begin{equation}
\min_{\tilde A \in \Sigma} \max_{i=1,\dots,n} \|a_i - \tilde a_i\|  = |\rho|.
\end{equation} 
See \cite{CC01,R95} for related discussions. 



\section{Gordan's Theorem with Margins}\label{sec:gordan}

We would like to make quantitative statements about what happens when either of the alternatives is satisfied \textit{easily} (with large positive or negative margin). Our preceding geometrical intuition suggests a refinement of Gordan's Theorem, namely Theorem~\ref{thm:gordan} below, that accounts for margins. Related results have been previously derived and discussed by Li and Terlaky~\cite{LT12} as well as by Todd and Ye~\cite{TY98}.  In particular, it can be shown that part 2 of Theorem~\ref{thm:gordan} could be obtained from \cite[Lemma 2.1 and Lemma 2.2]{TY98}.  Similarly, parts 2 and 3 could be recovered from \cite[Theorem 5 and Theorem 6]{LT12}.  We give a succinct and simple proof of Theorem~\ref{thm:gordan} by relying on  Proposition~\ref{margindual} and Proposition~\ref{thm:radthm}.  Theorem~\ref{thm:gordan} could also be proven, albeit less succinctly, via separation arguments  from convex analysis.

\begin{theorem}\label{thm:gordan}
For any problem instance $A$ and any constant $\gamma \geq 0$, 
\begin{enumerate}
\item Either $\exists w \in \Circle_A$ s.t. $A^T w > \zeros$, or  $\exists p\in \triangle$ s.t. $Ap=\zeros$.
\item Either $\exists w \in \Circle_A$ s.t. $A^T w > \gamma \ones$, or  $\exists p\in \triangle$ s.t. $\|Ap\| \leq \gamma$.
\item Either  $\exists w \in \Circle_A$ s.t. $A^T w > -\gamma \ones$, or $\forall v \in \gamma \CIRCLE_A$ \ $\exists p_v \in \triangle$ s.t. $v = Ap_v$.
\end{enumerate}
\end{theorem}
\begin{proof} The first statement is the usual form of Gordan's Theorem.  It is also a particular case of the other two when $\gamma = 0$.  Thus, we will prove the other two:
\begin{enumerate}
\setcounter{enumi}{1}
\item If the first alternative does not hold, then from the definition of $\rho_A$ it follows that $\rho_A \leq \gamma$.  In particular, $\rho_A^+ \le \gamma$.  To finish, observe that by Proposition~\ref{margindual} there exists $p \in \triangle$ such that 
\begin{equation}
\|Ap \| = \rho_A^+ ~\leq~ \gamma.
\end{equation}
\item Analogously to the previous case, if the first alternative does not hold, then $\rho_A \leq -\gamma$.  In particular, it captures 
\begin{equation}
|\rho_A^-| \ge \gamma.
\end{equation}
Observe that by Proposition~\ref{thm:radthm}, every point $ v\in \gamma \CIRCLE_A$ must be inside $\mathrm{conv}(A)$, that is, $v = Ap_v$ for some distribution $p_v\in \triangle$. 
\end{enumerate}
 One can similarly argue that in each case if the first alternative is true, then the second must be false. 
\end{proof}

In the spirit of radius theorems introduced in the previous section, the statements in Theorem~\ref{thm:gordan} can be equivalently written in the following succinct forms:
\begin{enumerate}
\item[1'.] Either $\{ w \in \Circle_A: A^T w > \zeros \} \neq \emptyset$, or $\zeros \in A\triangle$
\item[2'.] Either $\{w \in \Circle_A : A^T w > \gamma \ones \} \neq \emptyset$, or $\gamma \CIRCLE_A \cap A\triangle \neq \emptyset$
\item[3'.] Either $\{w \in \Circle_A: A^T w > -\gamma \ones \} \neq \emptyset$, or $\gamma \CIRCLE_A \subseteq A\triangle$ 
\end{enumerate}
As noted in the proof of Theorem~\ref{thm:gordan}, the first statement is a special case of the other two when $\gamma=0$.
In case 2, we have at least one witness $p$ close to the origin, and in 3, we have an entire ball of witnesses close to the origin.

\section{Hoffman's Theorem with Margins}\label{sec:hoffman}

Hoffman-style theorems are often useful to prove the convergence rate of iterative algorithms by characterizing the distance of a current iterate from a  target set. For example, a Hoffman-like theorem was also proved by \cite{HL12} (Lemma 2.3), where they use it to prove the linear convergence rate of the alternating direction method of multipliers, and in \cite{GPS12} (Lemma 4), where they use it to prove the linear convergence of a first order algorithm for calculating $\epsilon$-approximate equilibria in zero sum games. 

It is worth pointing out  that Hoffman, in whose honor the theorem is named and also an author of \cite{GHR95} whose proof strategy we follow in the alternate proof of Theorem \ref{thm:hoffpos}, himself appeared to have overlooked the intimate connection of the ``Hoffman constant'' ($\tau$ in Eq.(\ref{eq:hoffintro})) to the positive and negative margin, as we present in our theorems below.

\subsection{Hoffman's theorem for (D) when $\rho_A^- \neq 0$}

\begin{theorem}
Assume $A \in \R^{m\times n}$ is such that $|\rho_A^-| > 0$.  For $b\in \R^m$ define the ``witness'' set $W = \{x \geq \zeros | Ax = b\}$.  If $W \ne \emptyset$ then for all $x \geq \zeros$, 
\begin{equation}
\mathrm{dist}_1(x,W) ~\leq~ \frac{\|Ax-b\|}{|\rho_A^-|}
\end{equation}
where $\mathrm{dist}_1(x,W)$ is the distance from $x$ to $W$ measured by the $\ell_1$ norm $\|\cdot\|_1$.
\end{theorem}
\begin{proof}
Given $x \geq \zeros$ with $Ax \neq b$, consider a point
\begin{equation}
v ~=~ |\rho_A^-| \frac{b-Ax}{\|Ax-b\|}
\end{equation}
Note that $\|v\| = |\rho_A^-|$ and crucially $v \in \mathrm{lin}(A)$ (since $b \in \mathrm{lin}(A)$ since $W \neq \emptyset$). Hence, by Theorem \ref{thm:gordan}, there exists a distribution $p$ such that $v = Ap$.
Define
\begin{equation}
\bar x ~=~ x + p \frac{\|Ax-b\|}{|\rho_A^-|}
\end{equation}
Then, by substitution for $p$ and $v$ one can see that
\begin{equation}
A\bar x = Ax + v \frac{\|Ax-b\|}{|\rho_A^-|} = Ax + (b-Ax) = b
\end{equation}
Hence $\bar x \in W$, and $\mathrm{dist}_1(x,W) \leq \|x-\bar x\|_1 = \frac{\|Ax-b\|}{|\rho_A^-|}$. \end{proof}

The following variation (using witnesses only in $\triangle$) on the above theorem also holds.  This result is closely related to \cite[Lemma 2]{SP13} and has essentially the same proof.

\begin{proposition}\label{thm:hoffneg}
Assume $A \in \R^{m\times n}$ is such that $|\rho_A^-| > 0$. Define the set 
of witnesses $W = \{p \in \triangle | Ap = \zeros\}$.
Then at any $p \in \triangle$,
\begin{equation}
\mathrm{dist}_1(p,W)~\leq~\frac{2 \|Ap\|}{\|Ap\|+|\rho_A^-|}\le
\frac{2 \|Ap\|}{|\rho_A^-|} = \frac{ 2 \|p\|_G}{|\rho_A^-|}. 
\end{equation}
\end{proposition}
\begin{proof}
Assume $Ap \ne 0$ as otherwise there is nothing to show. Consider 
$v:= -\frac{|\rho_A^-|}{\|Ap\|} Ap.$  
Since $v\in\text{lin}(A)$ and $\|v\| = |\rho_A^-|$, Proposition~\ref{thm:radthm} implies that $v = Ap'$ for some $p'\in\triangle$.  Thus for $\lambda:=\frac{\|Ap\|}{\|Ap\|+|\rho_A^-|}$ we have $\tilde p:=\lambda p' + (1-\lambda)p\in W$ and 
\[
\mathrm{dist}_1(p,W) \le \|p-\tilde p\|_1 = \lambda \|p - p'\|_1 \le 2\lambda =
\frac{2 \|Ap\|}{\|Ap\|+|\rho_A^-|}=\frac{2 \|Ap\|}{|\rho_A^-|} = \frac{ 2 \|p\|_G}{|\rho_A^-|}.
\]
\end{proof}



\subsection{Hoffman's theorem for (P)  when $\rho_A^+ \neq 0$}

\begin{theorem}\label{thm:hoffpos}
Define $S = \{ y | A^T y \geq c \}$ for some vector $c$. Then, for all $w \in \mathbb{R}^d$,
\begin{equation*}
\mathrm{dist}(w,S) \leq \frac{\|[A^T w - c ]^-\|_\infty}{\rho_A^+}
\end{equation*}
where $\mathrm{dist}(w,S)$ is the $\|\cdot\|_2$-distance from $w$ to $S$ and $(x^-)_i = \min\{x_i,0\}$. 
\end{theorem}
\begin{proof}
Since $\rho_A^+ > 0$, the definitions of margin~\eqref{eq:margin} and affine margin~\eqref{affmargin} imply that there exists $\bar w \in \Circle_A$ with $A^T\bar w \ge \rho_A^+\ones$. Suppose, $A^T w \not\ge c$. Then we can add a multiple of $\bar w$ to $w$ as follows. Let $a = [c-A^T w]^+ = -[A^T w -c]^-$ where $(x^+)_i = \max\{x_i,0\}$ and $(x^-)_i = \min\{x_i,0\}$.  Since $a \ge c - A^Tw$ and $a \ge 0$, we have $\|a\|_\infty \ones \ge c-A^Tw$ and consequently
\begin{equation*}
A^T \left(w + \frac{\|a\|_\infty}{\rho_A^+}\bar w \right) ~\ge~ A^T w + \|a\|_\infty \ones \ge A^T w + (c-A^T w)= c .
\end{equation*}
Hence, $w + \frac{\|a\|_\infty}{\rho_A^+}\bar w \in S$ whose distance from $w$ is precisely $\frac{\|a\|_\infty}{\rho_A^+}$. 
\end{proof}

The interpretation of the preceding theorem is that the distance to feasibility for the problem (P) is governed by the magnitude of the largest mistake and the positive affine-margin of the problem instance $A$. 

We also provide an alternative proof of the theorem above, since proving the same fact from completely different angles can often yield insights. We follow the techniques of \cite{GHR95}, though we significantly simplify it. This is perhaps a more classical proof style, and possibly more amenable to other bounds not involving the margin, and hence it is instructive for those unfamiliar with proving these sorts of  bounds.


\begin{proof}[Alternate Proof of Theorem \ref{thm:hoffpos}]
For any given $w$, define $a = -(A^Tw-c )^- = (-A^T w+c )^+$ and hence note that $a \geq -(A^T w-c )$. 
\begin{eqnarray}
\min_{A^T u \geq c} \|w - u\| &=& \min_{A^T (u-w) \geq -A^Tw+c} \|w - u\| 
~=~ \min_{A^T z \geq -A^Tw+c} \|z\| \nonumber\\
&=& \sup_{\|\mu\| \leq 1} \left( \min_{A^T z \geq -A^Tw+c} \mu^T z \right) \label{eq:L2dual}\\
&=& \sup_{\|\mu\| \leq 1} \left ( \sup_{p \geq \zeros, Ap = \mu} p^T(-A^T w+c )\right ) \label{eq:LPdual}\\
&=& \sup_{\|p\|_G \leq 1, p\geq \zeros} p^\top (-A^T w+c)\label{eq:2G}\\
&\leq& \sup_{\|p\|_G \leq 1, p\geq \zeros} p^T a
~\leq~ \sup_{\|p\|_G \leq 1, p\geq \zeros} \|p\|_1 \|a\|_\infty\label{eq:CB}\\
&=& \frac{\|a\|_\infty}{\rho^+_A} \nonumber 
\end{eqnarray}
We used the self-duality of $\|.\|$ in Eq.(\ref{eq:L2dual}), LP duality for Eq.(\ref{eq:LPdual}),   $\|Ap\| = \|p\|_G$ by definition for Eq.(\ref{eq:2G}), and Holder's inequality in Eq.(\ref{eq:CB}). 
The last equality
follows because $\frac1{\rho_A^+} ~=~ \max_{\|p\|_G = 1, p\geq \zeros} \|p\|_1$, since $\rho_A^+ = \inf_{p \geq \zeros, \|p\|_1 = 1} \|p\|_G$ by Proposition~\ref{margindual}.  
\end{proof}




\section{ The Perceptron Algorithm: New Insights }\label{sec:NP}
 The Perceptron Algorithm was introduced and analysed by \cite{B62,N62,R58} 
 to solve the primal problem (P), with many variants in the machine learning literature. 
For ease of notation throughout this section assume $A = \left[\begin{array}{ccc}a_1 & \cdots & a_n \end{array} \right] \in \R^{d\times n}$ and $\|a_i\|=1,\; i=1,\dots,d$.
 The classical algorithm starts with $w_0:= a_i$ for any $i$, and 
in iteration $t$ performs
\begin{flalign*}
&\mbox{(choose any mistake)} &a_i &~:~ w_{t-1}^T a_i \; \le \; 0. &\nonumber\\
&&w_{t} &~\leftarrow~ w_{t-1}+a_i.&
\end{flalign*}
A variant called Normalized Perceptron which, as we point out in Theorem~\ref{thm:NPmargin} below, is a subgradient method, only updates on the worst mistake, and tracks a normalized $w$ that which is a convex combination of $a_i$'s.
\begin{flalign*}
&\mbox{(choose the worst mistake)} &a_i &~=~ \arg\min_{a_i} \{w_{t-1}^T a_i \}&\nonumber\\ 
&&w_{t} &~\leftarrow~ \Big (1-\tfrac1{t} \Big )w_{t-1}+\Big(\tfrac1{t} \Big)a_i.&
\end{flalign*}

The best known property of the unnormalized Perceptron or the Normalized Perceptron algorithm is that when (P) is strictly feasible with margin $\rho_A^+$, it finds such a solution $w$ in $1/\rho_A^{+2}$ iterations, as proved by \cite{N62,B62}. What is less obvious is that the Perceptron is actually \textit{primal-dual} in nature, as stated in the following result of Li and Terlaky~\cite{LT12}.  In the following statement by an {\em $\epsilon$-certificate for (D)} we mean a  vector $\alpha \in \triangle$ such that $\|A\alpha\| \le \epsilon.$ 

\begin{proposition}\label{prop:primdual}
If (D) is feasible, the Perceptron algorithm (when normalized) yields an $\epsilon$-certificate $\alpha_t$ for (D) in $1/\epsilon^2$ steps.  
\end{proposition}

Proposition~\ref{prop:primdual} and Proposition~\ref{thm:hoffneg} readily yield the following result.

\begin{corollary}\label{the.corollary}
Assume (D) is feasible and $|\rho_A^-| > 0$.  Define the set of witnesses $W=\{\alpha\in\triangle|A\alpha=0\}$.  If $w_t = A\alpha_t$ is the sequence of NP iterates then
\[
\mathrm{dist}_1(\alpha_t,W) \le \frac{2 }{|\rho_A^-|\sqrt{t}}
\]
\end{corollary}

\medskip

We prove two more nontrivial facts about the Normalized Perceptron that we have not found in the published literature for the case when (P) is feasible.
In this case not only does the Normalized Perceptron produce a \textit{feasible} $w$ in $O(1/\rho_A^{+2})$ steps, but on continuing to run the algorithm, $w_t$ will approach the \textit{optimal} $w$ that maximizes margin, i.e., achieves margin $\rho_A^+$. This is actually \textit{not} true with the classical Perceptron. 
The normalization in the following theorem is needed because $\|w_t\| \neq 1$. \\

\begin{theorem}\label{thm:NPmargin}
Assume (P) is feasible. If $w_t=A\alpha_t,\; t=0,1,\dots$ is the sequence of NP iterates with margin $\rho_t = \inf_{p \in \triangle} \frac{w_t}{\|w_t\|}^TAp$, and the optimal point  $w^* := \arg\sup_{\|w\|=1} \inf_{p \in \triangle} w^T A p$ achieves the optimal margin  $\rho = \rho_A^+ = \sup_{w\in\Circle}\inf_{p \in \triangle} w^{T}Ap$, then
\begin{equation*}
\rho^+_A - \rho_t ~\leq~ \Big \|\frac{w_t}{\|w_t\|} - w_* \Big\|  ~\leq~ \frac{4}{\rho_A^+ \sqrt t}.
\end{equation*}
\end{theorem}
\begin{proof}
Let $p_t := \arg\min_{p\in \triangle} w_t^TAp.$ Then
\[
\rho_A^+ - \rho_t = \inf_{p \in \triangle} w_*^{T}Ap - \frac{w_t}{\|w_t\|}^TAp_t \le 
\left(w_* - \frac{w_t}{\|w_t\|}\right)^T A p_t \le \Big \|\frac{w_t}{\|w_t\|} - w_* \Big\| \|Ap_t\| \le \Big \|\frac{w_t}{\|w_t\|} - w_* \Big\|.
\]
The last step because $\|a_i\|=1,\; i=1,\dots,n$ and $p\in\triangle.$

For the second inequality, first observe that
\begin{eqnarray}
\Big \|\frac{w_t}{\|w_t\|} - w_* \Big\| &=& \frac1{\|w_t\|} \Big \| w_t - \rho_A^+ w_* + (\rho_A^+ - \|w_t\|)w_*\Big \|\nonumber\\
&\leq& \frac1{\|w_t\|} \Big ( \|w_t - \rho_A^+ w_*\| + |\rho_A^+ - \|w_t\||\Big ) \nonumber \\
&\leq& \frac1{\rho_A^+} \Big ( \|w_t - \rho_A^+ w_*\| + |\rho_A^+ - \|w_t\||\Big )\label{eq:RSS2}
\end{eqnarray}
where the first inequality follows by the triangle inequality, and because $\|w_*\|=1$. The second inequality holds because $\rho_A^+ = \inf_{p \in \triangle} \|Ap\| $ and $\alpha_t \in \triangle$ implies that
\begin{equation}
\|w_t\| = \|A\alpha_t\| ~\geq~ \rho_A^+.\label{eq:lb}
\end{equation} 
The rest of the proof hinges on the fact that NP can be interpreted as a subgradient algorithm for the following problem: 
\begin{equation}\label{eq:NPSGD}
\min_{w \in \mathbb{R}^m}L(w) := \min_{w \in \mathbb{R}^m} \left(\half \|w\|^2 -\min_i  \{w^T a_i\} \right).
\end{equation}

We reproduce a short argument from  \cite{RP14,SP13} which shows that $L(w)$ is minimized at $\rho_A^+ w_*$. Let $\arg\min_\alpha L(w) = tw'$ for some $\|w'\|=1$ and some $t\in\R$. Substituting this into Eq.\eqref{eq:NPSGD}, we see that
\begin{equation*}
\min_{w \in \R^m} L(w) = \min_{t > 0} \{\half t^2 -t\rho_A^+ \} = -\tfrac1{2} \rho_A^{+2}
\end{equation*}
achieved at $t=\rho_A^+$ and $w'=w_*$. Hence $\arg\min_w L(w) = \rho_A^+ w_*$.

Note that the $(t+1)$-th iteration in the NP algorithm can be written as
\[
w_{t+1} = w_t - \frac{1}{t+1} g_t,
\]
for $ g_t = w_t - \arg\min_{a_i}\{w_t^Ta_i\} \in \partial L(w_t)$. 
Hence, the NP algorithm is a subgradient method for \eqref{eq:NPSGD}.  By construction, $L(w)$ is a 1-strongly convex function.  Since it is minimized at $\rho_A^+ w^*$, it follows that
\[
g_t^T(w_t - \rho_A^+w_*) \ge L(w_t) - L(\rho_A^+w_*)+\frac{1}{2}\|w_t - \rho_A^+w^*\|^2
\ge \|w_t - \rho_A^+w^*\|^2.\]
In addition, $\|g_t\| \le \|w_t\|(1 + \|a_i\|) \le 2 \|A\alpha_t\|  \le 2,$ so
\begin{align*}
\|w_{t+1} - \rho_A^+w_*\|^2 &= \left\| w_t - \frac{1}{t+1}g_t - \rho_A^+w_* \right\|^2 \\
& = \|w_{t} - \rho_A^+w_*\|^2  - \frac{2}{t+1}g_t^T(w_{t} - \rho_A^+w_*) + \frac{1}{(t+1)^2}\|g_t\|^2\\
&\le \left(1-\frac{2}{t+1} \right)\|w_{t} - \rho_A^+w_*\|^2 + \frac{4}{(t+1)^2}.
\end{align*}
It thus follows by induction on $t$ that
\begin{eqnarray}
\|w_t - \rho_A^+ w_*\| &\leq& 2/\sqrt t \nonumber \\ 
~\Rightarrow~ \|w_t\|- \rho_A^+ &\leq& 2/\sqrt t.\label{eq:ub}
\end{eqnarray}
This yields the required bound of $\frac{4}{\rho_A^+ \sqrt t}$ when plugged into Eq.(\ref{eq:RSS2}).
\end{proof}
Let us revisit the primal-dual formulation \eqref{eq:MEB} of the minimum enclosing ball problem.  The center of the minimum enclosing ball is precisely $c_*=\rho_A^+w_*$. Consequently the following result readily follows.

\begin{corollary}\label{cor.2} The sequence $w_t = A \alpha_t, \; t=0,1,\dots$ of NP iterates converges to the center  $c_* = \rho_A^+w_*$ of the  minimum enclosing ball problem \eqref{eq:MEB}.
\end{corollary}

The Normalized Perceptron algorithm also gives for free an estimate of $\rho_A^+$.

\begin{proposition}\label{prop:approxmargin} 
The Normalized Perceptron gives an $\epsilon$-approximation to the \textit{value} of the positive margin in $4/\epsilon^2$ steps. Specifically,
$$
\|w_{4/\epsilon^2}\| - \epsilon ~\leq~ \rho_A^+ ~\leq~ \|w_{4/\epsilon^2}\|
$$
\end{proposition}
\begin{proof}
The proof follows from Eq.\eqref{eq:ub} and Eq.\eqref{eq:lb}, which imply that $w_t$ satisfies $$\rho_A^+ \leq \|w_t\| \leq \rho_A^+ + 2/\sqrt t$$
whose rearrangement with $t=4/\epsilon^2$ completes the proof. 
\end{proof} 

It is worth noting that in sharp contrast to the estimate on $\rho_A^+$ given by Proposition~\ref{prop:approxmargin}, the question of finding elementary algorithms to estimate $|\rho_A^-|$ remains open.

\section{Discussion}

\subsection{Von-Neumann or Gilbert Algorithm for (D)}\label{sec.VNG} Von-Neumann described an iterative algorithm for solving dual (D) in a private communication with Dantzig in 1948, which was subsequently analyzed by the latter, but only published in \cite{D92}, and goes by the name of Von-Neumann's algorithm in optimization circles. Independently,  Gilbert~\cite{G66} described an essentially identical algorithm that goes by the name of Gilbert's algorithm in the computational geometry literature. We respect the independent findings in different literatures, and refer to it as the Von-Neumann-Gilbert (VNG) algorithm.  It starts from a point in conv($A$), say $w:=a_1$ and loops:
\begin{flalign*}
&\text{(choose furthest point)} &a_i &~=~ \arg\max_{a_i}\{\|w_{t-1}-a_i\|\} \\&\text{(line search, $\lambda \in [0,1]$)} &w_t &~\leftarrow ~ \arg\min_{w_\lambda} \|w_\lambda\|; ~ w_\lambda = \lambda w_{t-1} + (1-\lambda)a_i 
\end{flalign*}

Dantzig's paper showed that the Von-Neumann-Gilbert (VNG) algorithm can produce an $\epsilon$-approximate solution ($p$ such that $\|Ap\|\leq \epsilon$) to (D) in $1/\epsilon^2$ steps, establishing it as a dual algorithm as conjectured by Von-Neumann. Though designed for (D), Epelman and Freund~\cite{EF00} proved that when (P) is feasible, VNG also produces a feasible $w$ in $1/\rho_A^{+2}$ steps and hence VNG is also primal-dual like the Perceptron (as proved in Proposition \ref{prop:primdual}).  It readily follows that Theorem \ref{thm:NPmargin}, Corollary~\ref{the.corollary}, Corollary~\ref{cor.2}, and Proposition \ref{prop:approxmargin} hold as well with the Von-Neumann-Gilbert algorithm in place of the Normalized Perceptron algorithm.

Nesterov was the first to point out in a private note to \cite{EF97} that VNG is a Frank-Wolfe algorithm for
\begin{equation}\label{eq:VNGFW}
\min_{p \in \triangle}~ \|Ap\|
\end{equation}
Note that Eq.(\ref{eq:NPSGD}) is a relaxed version of Eq.(\ref{affmargin}), and also that Eq.(\ref{eq:VNGFW}) and Eq.(\ref{affmargin}) are Lagrangian duals of each other as seen in Eq.\eqref{eq:dual}. In this light, it is not surprising that NP and VNG algorithms have such similar properties. Moreover, Bach~\cite{B12} recently pointed out the strong connection via duality between subgradient and Frank-Wolfe methods.

However, VNG possesses one additional property. Restating a result of \cite{EF00} -- if $|\rho_A^-| > 0$, then VNG has linear convergence.   We include a simple geometrical proof of this result.

\begin{proposition} Assume (D) is feasible, $A = \left[\begin{array}{ccc}a_1 & \cdots & a_n \end{array} \right] \in \R^{d\times n}$ with $\|a_i\|=1,\; i=1,\dots,d$, and $|\rho_A^-|>0$.  Then the iterates $w_t = A\alpha_t$ generated by the VNG algorithm satisfy
\[
\|w_{t+1}\| \le \|w_t\|\sqrt{1-|\rho_A^-|^2}, \; t=0,1,\dots
\]
In particular, the algorithm finds $w_t=A\alpha_t, \; \alpha_t\in\triangle$ with $\|w_t\|\le \epsilon $ in at most $O\Big( \frac1{|\rho_A^-|^2}\log\left(\frac1{\epsilon}\right) \Big)$ steps. 
\end{proposition}
\begin{proof}
Figure~\ref{fig:VNG} illustrates the idea of the proof.  Assume $w_t = A\alpha_t \in \text{lin}(A)\ne 0$ as otherwise there is nothing to show.  By the definition of affine margin, there must exist a point $a_i$ such that $\cos\alpha = \frac{w_t}{\|w_t\|} \cdot a_i \leq \rho_A^- $ or equivalently $|\cos \alpha| \geq |\rho_A^-|$. VNG sets $w_{t+1}$ to be the nearest point to the origin on the  line joining $w_t$ with $a_i$. Consider $\tilde w$ as the nearest point to the origin on a (dotted) line parallel to $a_i$ through $w_t$. Note $(\pi/2 - \beta) + \alpha = \pi$ (internal angles of parallel lines). Then, $\|w_{t+1}\| \leq \|\tilde w\| = \|w_t\|\cos \beta = \|w_t\| \sin \alpha = \|w_t\|\sqrt{1-\cos^2\alpha} \leq \|w_t\|\sqrt{1-|\rho_A^-|^2}$.
\end{proof}
Hence, VNG can converge linearly with strict infeasibility of (P), but NP cannot. Nevertheless, NP and VNG can both be seen geometrically as trying to represent the center of circumscribing or inscribing balls (in (P) or (D)) of conv(A)  as a convex combination of input points. 

\subsection{Summary}

In this paper, we advance and unify our understanding of margins through a slew of new results and connections to old ones. First, we point out the correctness of using the affine margin, deriving its relation to the smallest ball enclosing conv(A), and the largest ball within conv(A). We proved generalizations of Gordan's theorem, whose statements were conjectured using the preceding geometrical intuition. Using these tools, we then derived interesting variants of Hoffman's theorems that explicitly use affine margins. We ended by proving that the Perceptron algorithm turns out to be primal-dual, its iterates are margin-maximizers, and the norm of its iterates are margin-approximators.

Right from his seminal introductory paper in the 1950s, Hoffman-like theorems have been used to prove convergence rates and stability of algorithms. Our theorems and also their proof strategies can be very useful in this regard, since such Hoffman-like theorems can be very challenging to conjecture and prove (see \cite{HL12} for example). Similarly, Gordan's theorem has been used in a wide array of settings in optimization, giving a precedent for the possible usefulness of our generalization. Lastly, large margin classification is now such an integral machine learning topic, that it seems fundamental that we unify our understanding of the geometrical, analytical and algorithmic ideas behind margins.

\subsection*{Acknowledgements}
This research was partially supported by NSF grant CMMI-1534850.


\newpage
\appendix

\section{Figures}\label{appsec:figures}
\begin{figure}[h!]
\begin{center}
\includegraphics[width = 0.3\linewidth]{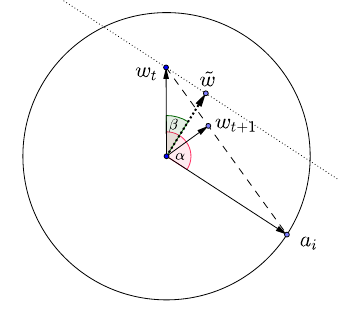}
\caption{Geometric illustration of a VNG iteration.}
\label{fig:VNG}
\end{center}
\end{figure}

\begin{figure}[h]
\begin{center}
\includegraphics[width = 0.75\linewidth]{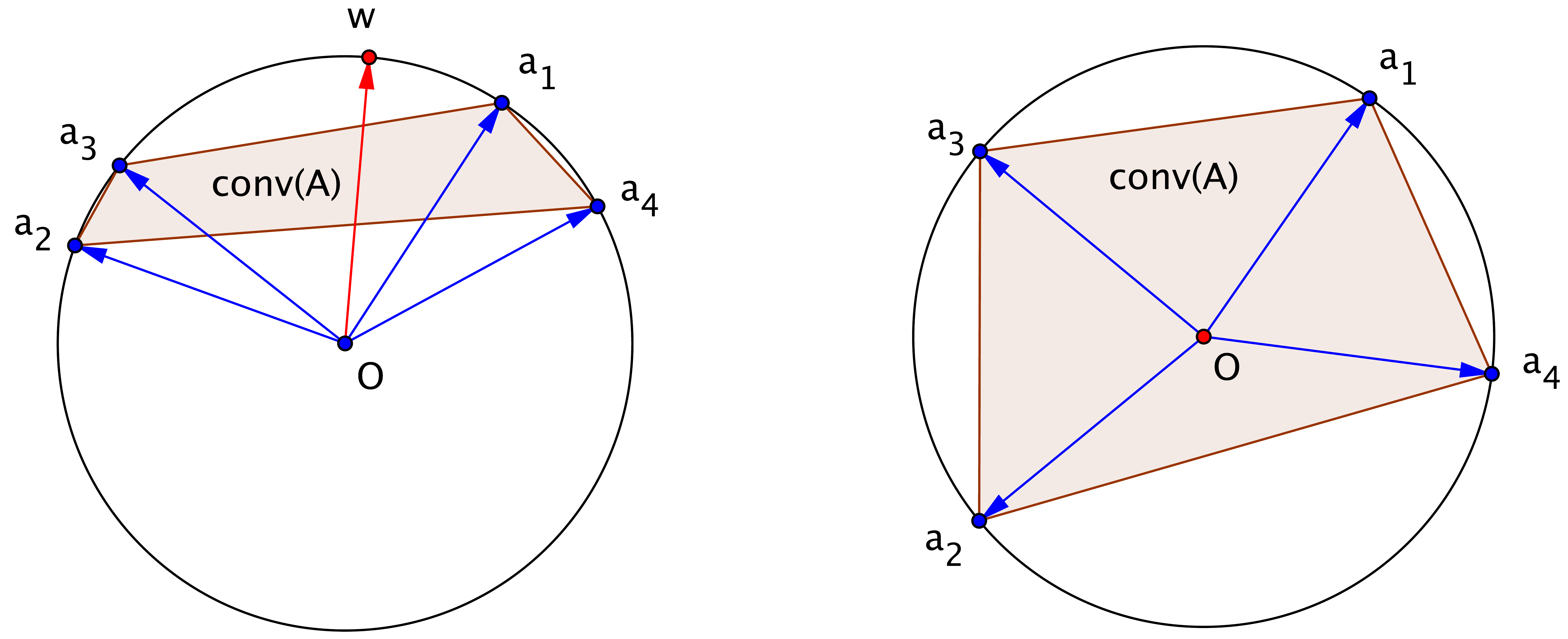}
\caption{Gordan's Theorem: Either there is a $w$ making an acute angle with all points, or the origin is in their convex hull. (note $\|a_i\|=1$)}
\label{fig:Gordan}
\end{center}
\end{figure}

\begin{figure}[h]
\begin{center}
\includegraphics[width = 0.75\linewidth]{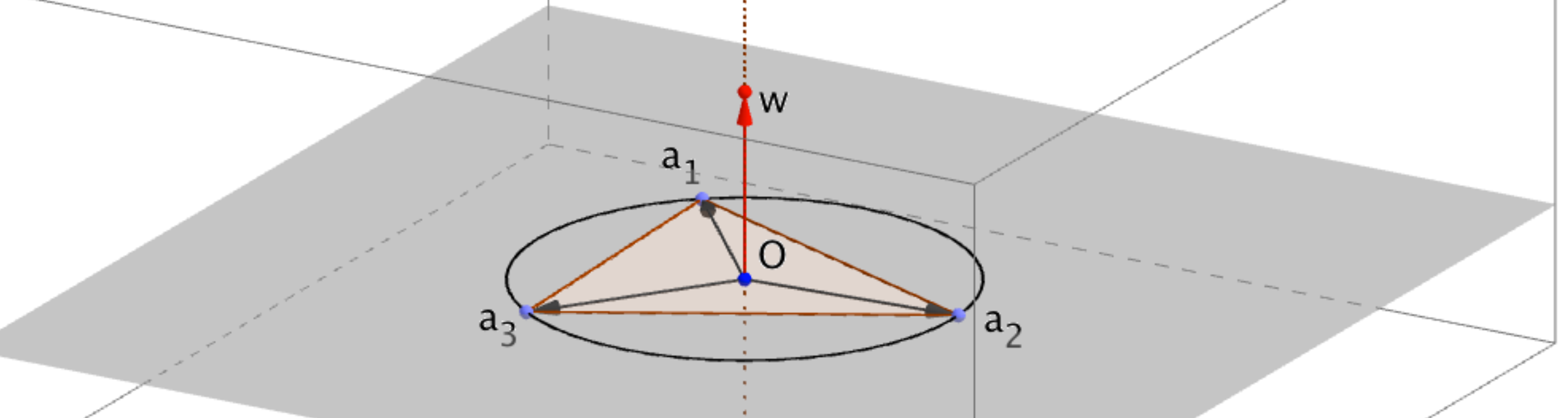}
\caption{When restricted to lin$(A)$, the margin is strictly negative. Otherwise, it would be possible to choose $w$ perpendicular to lin$(A)$, leading to a zero margin.}
\label{fig:affmargin}
\end{center}
\end{figure}

\begin{figure}[h!]
\begin{center}
\includegraphics[width = 0.35\linewidth]{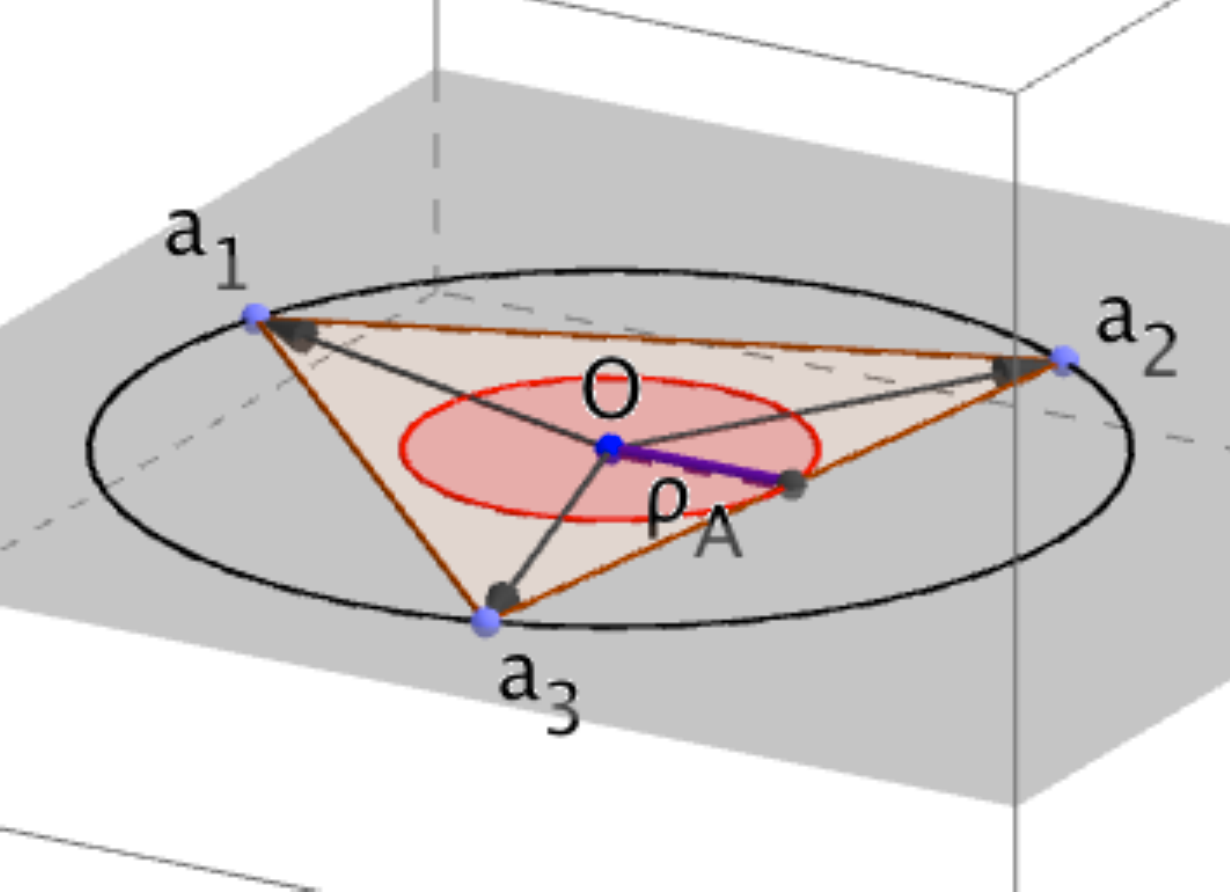} \hspace{0.3in}
\includegraphics[width = 0.35\linewidth]{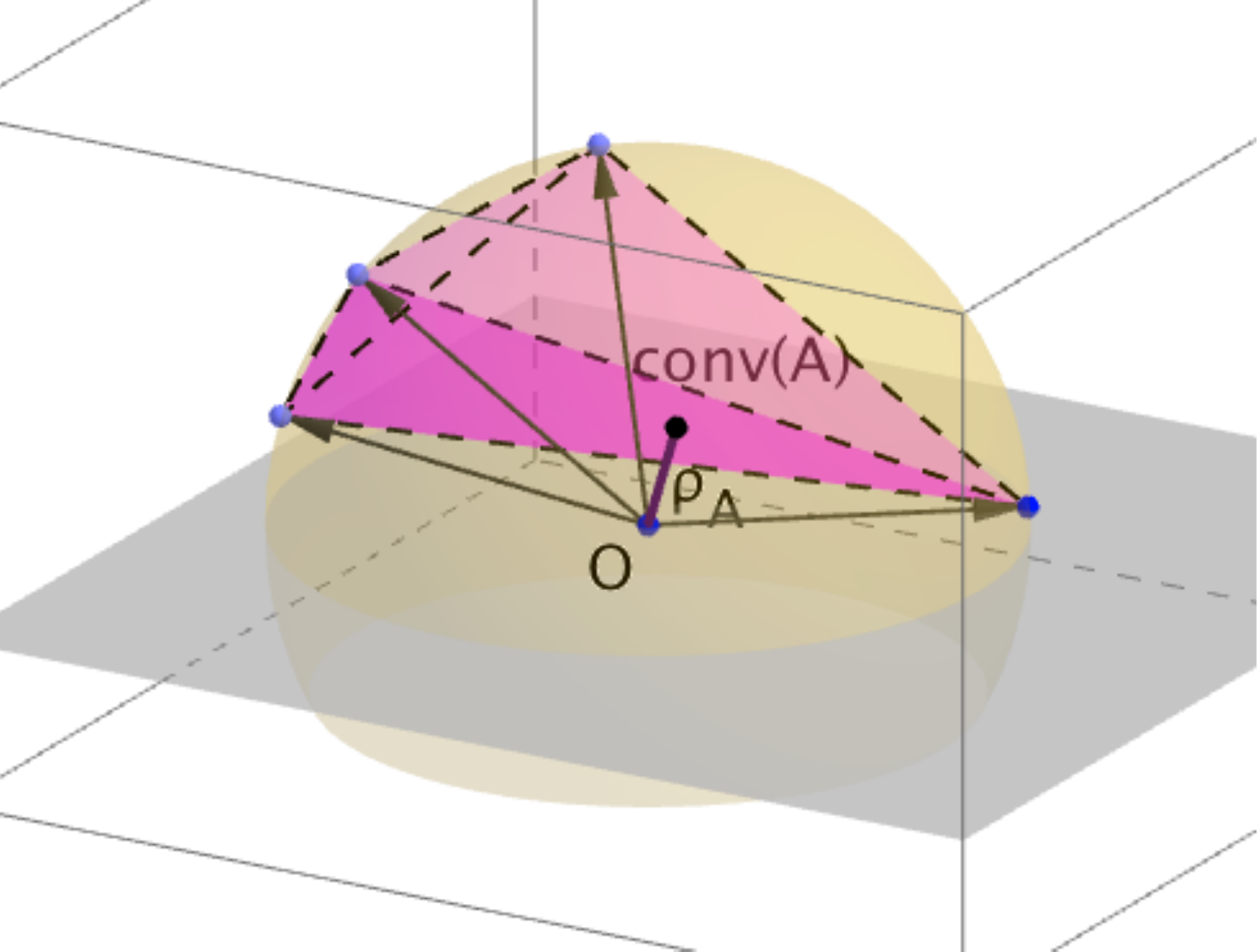}
\caption{Left: $\rho_A^-$ is the radius of the largest ball centered at origin, inside the relative interior of conv$(A)$. Right: $\rho_A^+$ is the distance from origin to conv$(A)$.
}
\label{fig:posnegmargin}
\end{center}
\end{figure}

\end{document}